\documentclass[12pt,letterpaper]{amsart}
\usepackage{geometry}
\usepackage{hyperref,amsmath}
\usepackage{mathpazo}
\usepackage{amssymb,url,setspace,xcolor}
\newtheorem{theorem}{Theorem}
\newtheorem{definition}{Definition}
\newtheorem{proposition}{Proposition}
\newtheorem{lemma}{Lemma}
\newtheorem{corollary}{Corollary}

\onehalfspace

\title[]{Compactness of composition operators on the Bergman spaces of convex domains and analytic discs}

\author{Timothy G. Clos}
\address{Timothy G. Clos,\, Kent State University}
\email{ tclos@kent.edu}
\date{\today}

\begin{document}

 \maketitle

\begin{abstract}
  We study the compactness of composition operators on the Bergman spaces of certain bounded pseudoconvex domains in $\mathbb{C}^n$ with non-trivial analytic discs contained in the boundary.  As a consequence we characterize that compactness of the composition operator with a continuous symbol (up to the closure) on the Bergman space of the polydisc.
\end{abstract}

\section{Introduction}
Let $\Omega\subset \mathbb{C}^n$ be a bounded convex domain.  Let $\mathcal{O}(\Omega)$ be the set of all holomorphic functions from $\Omega$ into $\mathbb{C}$.  Let $V$ be the Lebesgue volume measure on $\Omega$.  For $p\in [1,\infty)$ we define
\[A^p(\Omega):=\{f\in \mathcal{O}(\Omega):\int_{\Omega}|f|^p dV<\infty\}\] to be the $p$-Bergman space.  We denote the norm
as \[\|f\|_{p,\Omega}:=\left(\int_{\Omega}|f|^p dV\right)^{\frac{1}{p}}.\]  This paper will consider the $2$-Bergman space, which for brevity is denoted as the Bergman space.
Let $\phi:\Omega\rightarrow \Omega$ be a holomorphic self-map on $\Omega$.  That is, holomorphic in each coordinate function. Then we define the composition operator with symbol $\phi$ as
\[C_{\phi}(f)=f\circ \phi\] for all $f\in A^p(\Omega)$.  A set $\Delta\subset \mathbb{C}^n$ is called an analytic disc if there exists functions $f_j:\overline{\mathbb{D}}\rightarrow \mathbb{C}$ continuous on $\overline{\mathbb{D}}$ and holomorphic on $\mathbb{D}$ for $j\in \{1,2,...,n\}$ and $F(\mathbb{D})=\Delta$ where $F=(f_1,...,f_n)$.  If $\Delta$ is not a single point, then $\Delta$ is said to be a non-trivial analytic disc.  We define the collection of all non-trivial analytic discs in $b\Omega$ to be \[\mathcal{L}:=\bigcup\{F(\mathbb{D})|\,\,F:\mathbb{D}\rightarrow b\Omega\,,\text{holomorphic and non constant}\}.\] 

As we shall see, analytic structure in the boundary of the domain will play a crucial role in the proof of the main result.  The main theorem does not generalize easily to bounded pseudoconvex domains because of the special geometry of non-trivial analytic discs in the boundary of the domain.  Namely the following lemma states that on convex domains, non-trivial analytic discs have a nice form.  See also \cite{FuSt}.

\begin{lemma}[\cite{cucksahut}]\label{affine}
Let $\Omega\subset \mathbb{C}^n$ be a bounded convex domain.  Let $\Delta\subset b\Omega$ be a non-trivial analytic disc. Then there exists a complex line $L$ so that the convex hull of $\overline{\Delta}$ is contained in $L$. 
\end{lemma}

The following definitions are needed for the statement of the main results.  We first consider the notion of an angular derivative, which play an important role in the theory of composition operators in one and several variables. Here, $d(.,b\Omega)$ is Euclidean distance to the boundary.

\begin{definition} 

Let $\Omega\subset \mathbb{C}^n$ be a bounded domain for $n\geq 1$.  Let $\phi:\Omega\rightarrow \mathbb{C}^n$ be a holomorphic self-map.  Then for $p\in b\Omega$, we define the angular derivative at $p$ as
\[\lim_{\zeta\in \Omega, \, \zeta\rightarrow p}\frac{d(\phi(\zeta),b\Omega)}{d(\zeta,b\Omega)},\] provided this limit is finite. 
If \[\lim_{\zeta\in \Omega, \, \zeta\rightarrow p}\frac{d(\phi(\zeta),b\Omega)}{d(\zeta,b\Omega)}=\infty,\] we say the angular derivative does not exist at $p$.
\end{definition}

Of course, non-existence of the angular derivative at $p$ is equivalent to 
\[\lim_{\zeta\in \Omega, \, \zeta\rightarrow p}\frac{d(\zeta,b\Omega)}{d(\phi(\zeta),b\Omega)}=0.\]  Please see 
\cite{shapiro} for more information on angular derivatives and composition operators in general.

\begin{definition}
A bounded domain $\Omega\subset \mathbb{C}^n$ satisfies the limit point disc condition if $p\in \overline{\mathcal{L}}$ then there exists a non-trivial analytic disc $\Delta\subset b\Omega$ so that $p\in \overline{\Delta}$.
\end{definition}

It is worth noting there are domains that do not satsify the limit point disc condition.  In fact, any bounded domain $\Omega\subset\mathbb{C}^n$ satisfying these properties does not satisfy the limit point disc condition.

\begin{enumerate}

\item There exists a sequence of non-trivial analytic discs $\{\Delta_j\}_{j\in \mathbb{N}}$ so that $\Delta_j\subset b\Omega$ for all $j\in \mathbb{N}$.

\item $\overline{\Delta_{J}}\cap \overline{\Delta_{K}}=\emptyset$ for any $J,K\in \mathbb{N}$.

\item There exists $F_j:\overline{\mathbb{D}}\rightarrow b\Omega$ holomorphic on $\mathbb{D}$ and continuous on $\overline{\mathbb{D}}$ so that $F_j(\mathbb{D})=\Delta_j$ and $F_j\rightarrow (p_1,...,p_n)\in \mathbb{C}^n$ as $j\rightarrow \infty$. 

\end{enumerate}  
A bounded pseudoconvex domain $\Omega\subset \mathbb{C}^n$ with $C^2$-smooth boundary has defining function $\rho\in C^2(\mathbb{C}^n)$.  Then for $p\in b\Omega$, the Levi form at $p$ is the complex Hessian
\[\left(\frac{\partial^2 \rho}{\partial z_k\partial\overline{z_j}}(p)\right)_{n\times n}\] restricted to the complex tangent space at $p$.  See \cite{straube} for more information about the Levi form and its important properties.  For example, the Levi form is positive semi-definite on $b\Omega$.  Furthermore, if $p$ is in the closure of a non-trivial analytic disc in $b\Omega$, then $p$ is Levi-flat (the Levi form has a zero eigenvalue at $p$).  Let $\mathcal{F}$ denote the set of all Levi-flat points in $b\Omega$.  The following differential geometric condition involving the Levi form allows us to control the Levi-flat points.

\begin{definition}
Let $\Omega\subset \mathbb{C}^n$ be a bounded $C^2$-smooth convex domain.  We say $\Omega$ has the Levi-flat disc condition if 
\[\bigcup\{\overline{F(\mathbb{D})}|\,\,F:\mathbb{D}\rightarrow b\Omega\,,\text{holomorphic and non constant}\}=\mathcal{F}.\] 

\end{definition}

\begin{lemma}\label{flat}
Let $\Omega\subset \mathbb{C}^n$ be a bounded $C^2$-smooth convex domain.  If $\Omega$ has Levi-flat disc condition, then $\Omega$ satisfies the limit point disc condition.
\end{lemma}

\begin{proof}
Suppose $\Omega$ has the Levi-flat disc condition.  Then for any $q\in \overline{\mathcal{L}}$, there exists a sequence of points $\{q_j\}_{j\in \mathbb{N}}\subset \mathcal{L}$ converging to $q$.  Now $q_j$ are all Levi-flat, and the eigenvalues of the Levi form vary continuously.  This implies $q\in \mathcal{F}=\bigcup\{\overline{F(\mathbb{D})}|\,\,F:\mathbb{D}\rightarrow b\Omega\,,\text{holomorphic and non constant}\}$ by the Levi-flat disc condition.
\end{proof}

Note one can construct a smooth bounded convex Reinhardt domain with non-trivial analytic discs in the boundary satisfying the Levi-flat disc condition.  Intuitively, one can take a bidisc in $\mathbb{C}^2$ and 'smooth out' the distinguished boundary into a strongly pseudoconvex part.  See \cite{SC} and \cite{closdis} for more information on convex Reinhardt domains in $\mathbb{C}^2$.

\section{Some Background and Main Results}
Compactness of composition operators was studied on the unit disc in $\mathbb{D}$ in the article \cite{ghatagecompact}. Here, the authors of \cite{ghatagecompact} study the angular derivative of the symbol near the boundary and obtain a compactness result.  They then construct a counterexample to show that the converse of their theorem does not hold true.  The authors of 
\cite{ghatageclosed} studied the closed range property of composition operators on the unit disc.  Work on essential norm estimates and compactness of composition operators was studied on the ball in $\mathbb{C}^n$ and on the unit disc in $\mathbb{C}$ by \cite{CowenMacCluer} and \cite{HMW}.  On more general bounded strongly pseudoconvex domains in $\mathbb{C}^n$, \cite{cuckruhanseveral} studied the essential norm of the composition operator in terms of the behavior of the norm of the normalized Bergman kernel composed with the symbol.   

 In the work \cite{zhu}, Zhu uses the normalized Bergman kernel as a weakly convergent sequence.  Since these can be explicitly computed on the ball, this results in a non-trivial characterization of compactness.  Zhu in \cite{zhu} also extends these results to bounded strongly pseudoconvex domains.  Both these characterizations involve the angular derivative mentioned earlier.  In the absence of such explicit Bergman kernels (on general bounded convex domains), we construct weakly convergent sequences using the convexity of the domain in a significant way, and use estimates on the normalized Bergman kernel near strongly pseudoconvex points.  Here, $K^{\Omega}_{z}$ denotes the Bergman kernel of $\Omega$, $V$ is the Lebesgue volume measure.  We let $k_z^{\Omega}$ be the normalized Bergman kernel of $\Omega$.  As notation, we sometimes write $k_z$ as the normalized Bergman kernel of $\Omega$ if there is no ambiguity about the domain.
The following result is well known in the literature and a proof is given in \cite{ak}.
\begin{proposition}\label{propkernel}
Let $\Omega$ be a smooth bounded pseudoconvex domain.  Then the normalized Bergman kernel 
$k_{\zeta}^{\Omega}\rightarrow 0$ weakly in $A^2(\Omega)$ as $\zeta\rightarrow b\Omega$.
\end{proposition}

The following two thereoms are the main results.
This first result states that if $C_{\phi}$ is compact, $\Omega$ is convex, and $\phi\in C(\overline{\Omega})$ then $\phi(b\Omega)$ stays away from $\overline{\mathcal{L}}$.  We note that this limit point disc condition is an important geometric assumption, as it allows us to construct a weakly converging sequence (in the Bergman space) with the required bound from below near a non-trivial analytic disc in the boundary.  

\begin{theorem}\label{analytic disc}
Let $\Omega\subset \mathbb{C}^n$ be a bounded convex domain satisfying the limit point disc condition.  Let $\phi:\Omega\rightarrow \Omega$ be a holomorphic self map, $\phi\in C(\overline{\Omega})$, and suppose $C_{\phi}$ is compact on $A^2(\Omega)$.  Then,
$d(\phi(b\Omega),\overline{\mathcal{L}})>0$.

\end{theorem}

Next we consider symbols that are holomorphic on a neighborhood of $\overline{\Omega}$ for a strongly convex bounded domain $\Omega$.  We let $J_{\phi}$ be the determinant of the Jacobian matrix of $\phi$.

\begin{theorem}\label{main}
  Let $\Omega\subset \mathbb{C}^n$ be a smooth bounded strongly convex domain.  Let $\phi:\Omega\rightarrow \Omega$ be a holomorphic self-map, where $\phi:=(\phi_1,...,\phi_n)$ and $\phi_j$ are holomorphic on a neighbourhood of $\overline{\Omega}$ for $j=1,...,n$.  If $C_{\phi}$ is compact on $A^2(\Omega)$, then $b\Omega\cap b\phi(\overline{\Omega})$ is nowhere dense in $b\Omega$.
\end{theorem}

If we assume that the domain is smooth, bounded, convex, and with a certain boundary condition controlling weakly pseudoconvex points, we can get a generalization of one implication of \cite{zhu}.  Recall \cite{zhu} showed that compactness of $C_{\phi}$ implies the angular derivative does not exist at any boundary point of a bounded strongly pseudoconvex domain (which has no non-trivial analytic discs in its boundary).  Furthermore, \cite{zhu} showed the non-existence of the angular derivative of $\phi$ at any boundary point of a bounded strongly pseudoconvex domain is a necessary and sufficent condition for compactness of $C_{\phi}$.  The following corollary is consequence of Theorem \ref{analytic disc}.  Notice we allow the existence of non-trivial analytic discs in the following corollary.

\begin{corollary}\label{cormain}
Let $\Omega\subset \mathbb{C}^n$ be a smooth bounded convex domain satisfying the Levi-flat disc condition.   Let $\phi:\Omega\rightarrow \Omega$ be a holomorphic self-map, and $\phi\in C(\overline{\Omega})$.  Suppose $C_{\phi}$ is compact on $A^2(\Omega)$, and $p\in b\Omega$ be a strongly pseudoconvex point.  Then the angular derivative does not exist at $p$.  That is,
\[  \lim_{\zeta\in \Omega\,\,,\zeta\rightarrow p}\frac{d(\zeta,b\Omega)}{d(\phi(\zeta),b\Omega)}=0.\]  

\end{corollary}

\begin{proof}
From Theorem \ref{analytic disc} and Lemma \ref{flat}, we may assume $\phi(p)\in b\Omega$ and the boundary of $\Omega$ near $\phi(p)$ is strongly pseudoconvex.  Then one can show that, for $z$ near $p$,
\[\|C^*_{\phi}(k_z)\|^2=K(z,z)^{-1}K(\phi(z),\phi(z))\approx \frac{d(z,b\Omega)^{n+1}}{d(\phi(z),b\Omega)^{n+1}}.\]  Since $C^*_{\phi}$ is compact and $k_z\rightarrow 0$ weakly as $z\rightarrow p$ (refer to Proposition \ref{propkernel}), we have 
\[  \lim_{\zeta\in \Omega\,\,,\zeta\rightarrow p}\frac{d(\zeta,b\Omega)}{d(\phi(\zeta),b\Omega)}=0.\]

\end{proof}

There are bounded convex domains with the Levi-flat disc condition (namely bounded strongly convex domains) where the converse of Corollary \ref{cormain} holds.  However, it remains an open question in general.  The following is a corollary of Theorem \ref{analytic disc} and Proposition \ref{necessary} and concerns the $n$-product of discs.

\begin{corollary}\label{poly}
Let $\mathbb{D}^n$ be the polydisc in $\mathbb{C}^n$.  Suppose $\phi\in C(\overline{\Omega})$ and $\phi:\mathbb{D}^n\rightarrow \mathbb{D}^n$ is a holomorphic self map.  Then $C_{\phi}$ is compact if and only if $\phi(\mathbb{D}^n)$ is relatively compact in $\mathbb{D}^n$.
\end{corollary}

\section{Proof of Main Results}

First we show that compactness of composition operators is invariant under biholomorphisms of the domain $\Omega$.

\begin{lemma}\label{lem2}
Let $\Omega_1,\Omega_2\subset \mathbb{C}^n$ for $n\geq 2$ be bounded pseudoconvex domains.  Furthermore, assume there exists a biholomorphism $B:\Omega_1\rightarrow \Omega_2$ so that $B\in C^1(\overline{\Omega_1})$.  Suppose $\phi:=(\phi_1,\phi_2,...,\phi_n):\Omega_2\rightarrow \Omega_2$ is such that the composition operator $C_{\phi}$ is compact on $A^2(\Omega_2)$.  Then, $C_{B^{-1}\circ \phi\circ B}$ is compact on $A^2(\Omega_1)$.  
\end{lemma}

\begin{proof}
Let $g_j\in A^2(\Omega_1)$ so that $g_j\rightarrow 0$ weakly as $j\rightarrow \infty$.  We will use the fact that $g_j\rightarrow 0$ weakly in $A^2(\Omega_1)$ as $j\rightarrow \infty$ if and only if $\|g_j\|$ is a bounded sequence in $j$ and $g_j\rightarrow 0$ uniformly on compact subsets of $\Omega_1$.  This fact appears as \cite[lemma 3.5]{cuckruhanseveral}. Therefore, $\|g_j\|$ is uniformly bounded in $j$ and $g_j\rightarrow 0$ uniformly on compact subsets of $\Omega_1$.
Then define $h_j:=g_j\circ B^{-1}\in  A^2(\Omega_2)$.  Then using a change of coordinates, one can show $\|h_j\|$ is uniformly bounded in $j$ and $h_j\rightarrow 0$ uniformly on compact subsets of $\Omega_2$.  Therefore, by \cite[lemma 3.5]{cuckruhanseveral}, $h_j\rightarrow 0$ weakly as $j\rightarrow \infty$.  
Then we have,
\begin{align*}
&\|C_{B^{-1}\circ \phi\circ B}(g_j)\|_{2,\Omega_1}^2\\
&=\|h_j\circ \phi\circ B\|_{2,\Omega_1}^2\\
&\leq \sup\{|J_{B^{-1}}(z)|^2:z\in \Omega_2\}\|C_{\phi}(h_j)\|_{2,\Omega_2}^2
\end{align*}

This shows that $C_{B^{-1}\circ\phi\circ B}$ is compact on $A^2(\Omega_1)$.

\end{proof}

\begin{proposition}\label{necessary}

Let $\Omega\subset \mathbb{C}^n$ be a bounded pseudoconvex domain.  Suppose $\phi$ is a holomorphic self-map on $\Omega$ so that $\overline{\phi(\Omega)}\subset \Omega$.  Then $C_{\phi}$ is compact on $A^2(\Omega)$.  

\end{proposition}

\begin{proof}
We let $V_{\phi}:=V\circ \phi^{-1}$ be the pullback measure associated with $\phi$, where $V$ is the Lebesgue volume measure on $\Omega$.  Since $\phi(\Omega)$ is relatively compact in $\Omega$, the support of $V_{\phi}$ is relatively compact in $\Omega$.  Let $\{g_j\}_{j\in \mathbb{N}}\subset A^2(\Omega)$ so that $g_j\rightarrow 0$ weakly as $j\rightarrow \infty$.  Thus $g_j\rightarrow 0$ uniformly on compact subsets of $\Omega$ as $j\rightarrow \infty$ (see \cite[lemma 3.5]{cuckruhanseveral}).  Since $V_{\phi}$ has compact support, there exists $M$ relatively compact in $\Omega$ so that 
\[\|C_{\phi}g_j\|^2_{L^2(\Omega, V)}=\|g_j\|^2_{L^2(\overline{M}, V_{\phi})}\leq \|g_j\|^2_{L^{\infty}(\overline{M})}V_{\phi}(\overline{M})\rightarrow 0\] as $j\rightarrow \infty$.

\end{proof}

\begin{proof}[Proof of Theorem \ref{analytic disc}]
Let us assume $C_{\phi}$ is compact on $A^2(\Omega)$ and $\phi^{-1}( b\Omega)\neq \emptyset$ (otherwise $d(\phi(b\Omega),\overline{\mathcal{L}})>0$ is vacuously true).  Let $p\in \phi^{-1}( b\Omega)$.
The first claim is that $\phi(p)$ cannot be contained in the closure of a non-trivial analytic disc in $b\Omega$.  For the sake of obtaining a contradiction, suppose $\phi(p)\in \overline{\mathcal{L}}\subset b\Omega$.  By the convexity of $\Omega$, we may assume $\Omega\subset \{Re(z_n)>0\}$.  Since $\Omega$ satisfies the limit point disc condition, we may assume there exists a non-trivial analytic disc $\Delta\subset b\Omega$ and $\phi(p)\in \overline{\Delta}$.  Then, by Lemma \ref{affine}, $\Delta\subset L$ for some complex line $L$.  Let $p_j\in \Omega$ so that $p_j\rightarrow p$ as $j\rightarrow \infty$.  Then we construct $g_j$ in a similar way to \cite[Proof of Theorem 2]{ClosSahutoglu}.  That is, by rotating and translating the domain and appealing to Lemma \ref{lem2}, we may assume $\phi(p)=(0,0,0,...,0)$, $\Omega\subset \{Re(z_n)>0\}$ and $L\subset \{Re(z_n)=0\}$.  We define
\[g_j(z_1,...,z_n):=\frac{|\phi_n(p_j)|^{k(n)}}{d(p_j,b\Omega)^{n}\left(z_n+\delta_j\right)^{k(n)}}.\] The condition that $\overline{\Omega}\subset \{Re(z_n)\geq 0\}$ implies $g_j$ is continuous on $\overline{\Omega}$ for all $j\in \mathbb{N}$.  Also, chose $k(n)\in \mathbb{N}$ is sufficiently large so that $g_j\rightarrow 0$ weakly in $A^2(\Omega)$ as $j\rightarrow \infty$.  One can see this via a continuity argument since one can chose $k(n)$ and $\delta_j>0$ so that $\{\|g_j\|_{L^2(\Omega)}\}_{j\in \mathbb{N}}$ is bounded and $g_j\rightarrow 0$ uniformly on compact subsets of $\Omega$.

We let 
\[B_j:=\{z\in \mathbb{C}^n: d(z,p_j)<\frac{1}{2}d(p_j,b\Omega) \}\subset \Omega\] and apply the mean value theorem for holomorphic functions.  We have,
\begin{align*}
&\int_{\Omega}|g_j\circ \phi(z)|^2dV(z)\\
&\geq\int_{B_j}|g_j\circ \phi(z)|^2dV(z)\\
&=\int_{B_j}\left|\frac{|\phi_n(p_j)|^{k(n)}}{d(p_j,b\Omega)^{n}\left(\phi_n(z)+\delta_j\right)^{k(n)}}\right|^2dV(z)\\
&\geq vol(B_j)\left|\frac{|\phi_n(p_j)|^{k(n)}}{d(p_j,b\Omega)^{n}\left(\phi_n(p_j)+\delta_j\right)^{k(n)}}\right|^2\\
&\gtrsim C
\end{align*}
For some $C>0$.  This contradicts the compactness of $C_{\phi}$.  Therefore, $\phi(p)$ cannot be contained in the closure of a non-trivial analytic disc.      
\end{proof}

\begin{proof}[Proof of Theorem \ref{main}]
We will show that if $p\in b\Omega\cap b\phi(\overline{\Omega})$, then $J_{\phi}(p)=0$ and use our assumption that $\Omega$ is strongly convex and hence has no non-trivial analytic discs in its boundary.  Assume that  $p\in b\Omega\cap b\phi(\overline{\Omega})$ and $J_{\phi}(p)\neq 0$.  Thus by the inverse function theorem, $\phi$ has a local inverse $\psi:\phi(U)\rightarrow U$ for some open set $U$ in $\mathbb{C}^n$ containing $p$.  Without loss of generality and appealing to Lemma \ref{lem2}, by rotating and translating the domain, we may assume that $\phi(p)=(0,0,...,0)$.  Let $g_j$ be defined as in the proof of Theorem \ref{analytic disc}.  One can show that for any open set $V$ containing $(0,...,0)$, there exists $\delta_V>0$ so that $\|g_j\|^2_{L^2(V\cap \Omega)}\geq \delta_V$ for all $j\in \mathbb{N}$.
Then we have
\[ \|C_{\phi}(g_j)\|^2_{L^2(U\cap \Omega)}=\||J_{\psi}|^{\frac{1}{2}}g_j\|^2_{L^2(\phi(U\cap \Omega))}=\||J_{\psi}|^{\frac{1}{2}}g_j\|^2_{L^2(V\cap \Omega)}\geq M_V>0\] for some open set $V$, some $M_V>0$, and for all $j\in \mathbb{N}$.  This is a contradiction.  Therefore, $J_{\phi}(p)=0$.  Now  $\overline{b\Omega\cap b\phi(\overline{\Omega})}$ must have empty interior in $b\Omega$ since $\overline{b\Omega\cap b\phi(\overline{\Omega})}$ contains the zero set of the Jacobian of $\phi$, a holomorphic function.  If $\overline{b\Omega\cap b\phi(\overline{\Omega})}$ has non-empty interior in $b\Omega$, then one can show that $b\Omega$ contains a non-trivial analytic disc.  This cannot occur since our domain is strongly pseudoconvex.
\end{proof}

This proof shows that the converse of Theorem \ref{analytic disc} is false.  Let $\Omega\subset \mathbb{C}^2$ be a smoothly bounded convex Reinhardt domain so that $\mathcal{L}\neq \emptyset$ but with no 'vertical'  non-trivial analytic discs (see \cite{SC} and \cite{closdis}).  Furthermore, assume $\Omega$ has the Levi-flat disc condition, $\mathcal{L}\subset \mathbb{C}\times \{|z_2|=1\}$, and $(1,0)\in b\Omega$.  Define $\phi(z_1,z_2):=(z_1, \frac{1}{2}z_2)$.  Since $\Omega$ is convex and Reinhardt, it is a complete Reinhardt domain (see \cite{SC} and \cite{closdis}).  Therefore, $\phi$ is a holomorphic self-map on $\Omega$ and $\overline{\phi(\Omega)}\cap \overline{\mathcal{L}}=\emptyset$ (so satisfies the assumptions of the converse of Theorem \ref{analytic disc}).  However, it is clear the $C_{\phi}$ is not compact, since $J_{\phi}$ is non-vanishing at $(1,0)$ and $\phi(1,0)=(1,0)$.  This next example will show the converse of Theorem \ref{main} is not true in general.  Let $\Omega=\mathbb{B}^2$ be the unit ball in $\mathbb{C}^2$, and let $\phi(z_1,z_2):=(z_1,\frac{1}{2}z_2)$ as before.  The claim is that $b\Omega\cap b\phi(\overline{\Omega})$ is nowhere dense in $b\Omega$ but $C_{\phi}$ is not compact on $A^2(\Omega)$, implying the converse of Theorem \ref{main} is not true.  Suppose $|z_1|^2+|z_2|^2=1$ and compute the magnitude of the image of $(z_1,z_2)$ under $\phi$.  If $|z_1|^2+\frac{1}{4}|z_2|^2=1$, then we get that $z_2=0$ and so $|z_1|=1$.  This implies $b\Omega\cap b\phi(\overline{\Omega})$ is nowhere dense in $b\Omega$.  It is clear that $C_{\phi}$ is not compact on $A^2(\Omega)$ by \cite{zhu} since the angular derivative of $\phi$ exists at $(1,0)$.  
The techniques used in the previous proof, along with the maximum principle for holomorphic functions and Sard's theorem, can be adapted to show the following.

\begin{proposition}

Let $\Omega\subset \mathbb{C}^n$ be  a smoothly bounded pseudoconvex domain.  Let $\phi:\Omega\rightarrow \Omega$ be a proper holomorphic map where $\phi\in C^1(\overline{\Omega} )$.  Then $C_{\phi}$ is not compact on $A^2(\Omega)$.

\end{proposition}

\begin{proof}
Assume $\phi$ is a proper holomorphic map on $\Omega$, $\phi$ is of class $C^1(\overline{\Omega})$, and $C_{\phi}$ is compact on $A^2(\Omega)$. Now if $\phi$ is proper, then $\phi:\overline{\Omega}\rightarrow \overline{\Omega}$ is surjective and $\phi(b\Omega)=b\Omega$.  Therefore, by the proof of Theorem \ref{main}, $J_{\phi}=0$ on $b\Omega$.  Since $J_{\phi}$ is holomorphic on $\Omega$ and continuous up to $\overline{\Omega}$, $J_{\phi}=0$ on $\Omega$.  This contradicts Sard's theorem (see \cite{sard}), since the image of the zero set of the Jacobian has positive volume measure.

\end{proof}

It would be interesting say whether these techniques generalize to symbols of less regularity.  Also, one can try to relate compactness of composition operators to angular derivatives (see \cite{zhu}) of symbols on more general bounded convex and bounded pseudoconvex domains.

\section{Aknowlegments}
I wish to thank S\"{o}nmez \c{S}ahuto\u{g}lu, Akaki Tikaradze, and Trieu Le for useful conversations and comments on a preliminary version of this manuscript.  I also thank the anonymous referees for their useful suggestions.

\bibliographystyle{amsalpha}
\bibliography{refscomp}

\end{document}